\documentclass[12pt]{amsart}

\usepackage{amsmath}
\usepackage{amssymb}
\usepackage{amsthm}
\usepackage{babel}
\usepackage{caption}
\usepackage{color}
\usepackage{dsfont}
\usepackage[mathcal]{eucal}
\usepackage{float}
\usepackage{fullpage}  
\usepackage[utf8]{inputenc}
\usepackage{mathrsfs}
\usepackage{mathtools}
\usepackage{soul,xcolor}
\usepackage{stmaryrd}
\usepackage{thmtools}
\usepackage{tikz-cd}
\usepackage{verbatim}  
\usepackage[all]{xy}
\usepackage[colorlinks,citecolor=blue,linkcolor=blue,urlcolor=blue,filecolor=blue,breaklinks]{hyperref}

\newcommand{\bC}{\mathbb{C}}

\newcommand{\bH}{\mathbb{H}}

\newcommand{\bR}{\mathbb{R}}

\newcommand{\bZ}{\mathbb{Z}}

\renewcommand{\phi}{\varphi}

\newcommand{\excise}[1]{}

\newcommand{\abs}[1]{\left \vert {#1} \right \vert}
\newcommand{\s}[1]{\left\{ {#1} \right\}}

\newcommand{\parens}[1]{\left( {#1} \right)}
\newcommand{\defeq}{\vcentcolon=}
\newcommand{\angles}[1]{\left \langle {#1} \right \rangle}

\newcommand{\hdist}[2]{\text{dist}_{\text{H}}\parens{{#1},{#2}}}
\newcommand{\Isom}[1]{\textup{Isom}({#1})}

\providecommand{\ar}{\rightarrow}

\newtheorem{theorem}{Theorem}
\newtheorem{thmx}{Theorem}

\newtheorem{proposition}[theorem]{Proposition}
\newtheorem{lemma}[theorem]{Lemma}
\newtheorem{corollary}[theorem]{Corollary}

\theoremstyle{definition}
\newtheorem{definition}[theorem]{Definition}

\newtheorem{example}[theorem]{Example}

\newtheorem{remark}[theorem]{Remark}

\title{Quasi-isometric rigidity for a product of lattices}
\author{Josiah Oh}

\date{\today}
\thanks{The first author was partially supported by the NSF, under grant DMS-1547357}

\begin{document}

\begin{abstract}
We demonstrate quasi-isometric rigidity for the product of a non-uniform rank one lattice and a nilpotent lattice. Specifically, we show that any finitely-generated group quasi-isometric to such a product is, up to finite noise, an extension of a non-uniform rank one lattice by a nilpotent lattice. Furthermore, we show under extra conditions that this extension is nilcentral, a notion which generalizes central extensions to extensions by a nilpotent group. 
\end{abstract}

\maketitle

\section{Introduction}

\subsection{Background}

One theme of geometric group theory is the rich relationship between the algebraic structure of groups and the large-scale geometry of spaces on which they act. By the ``large-scale geometry" of a space we mean the metric structure that is preserved by quasi-isometries, maps which preserve distances up to a controlled error. Gromov proposed in a 1983 ICM address \cite{gromov-icm} a broad research program of studying finitely generated groups as geometric objects and classifying them up to quasi-isometry. One aspect involves identifying instances of quasi-isometric rigidity, the phenomenon which occurs when algebraic properties of a finitely generated group are determined by its large-scale geometry. For example, a celebrated theorem by Gromov \cite{Gromov-polynomial} states that finitely generated groups of polynomial growth are virtually nilpotent (the converse is also true \cite{wolf1968growth}). Since growth rate is invariant under quasi-isometry, it follows that any finitely generated group which is quasi-isometric to a space of polynomial growth necessarily has a nilpotent subgroup of finite index. Thus we have an instance of when the quasi-isometry type of a group determines some of its algebraic structure.

One of the landmark results in research on quasi-isometric rigidity is the complete quasi-isometry classification of lattices in semisimple Lie groups. A large body of work in the 1980s and 1990s by several people over many papers culminated in a general theorem on the rigidity of the class of lattices among all finitely generated groups (see \cite{farb} for a detailed survey). Informally, this theorem states that any group quasi-isometric to a lattice in a semisimple Lie group is almost a lattice in that Lie group. More precisely,

\begin{theorem}[Rigidity of lattices]
    If $\Gamma$ is a finitely generated group quasi-isometric to an irreducible lattice in a semisimple Lie group $G$, then there is a short exact sequence
    \begin{equation*}
    \begin{tikzcd}
        1 \ar[r] & F \ar[r] & \Gamma \ar[r] & \Lambda \ar[r] & 1
    \end{tikzcd}
    \end{equation*}
    where $\Lambda$ is a lattice in $G$, and $F$ is a finite group.
\end{theorem}

One of the major breakthroughs leading to this general classification was the work of Schwartz \cite{Schwartz} on non-uniform lattices in rank one semisimple Lie groups. These Lie groups agree, up to index 2, with the isometry groups of the negatively curved symmetric spaces: real, complex, quaternionic hyperbolic space, and the Cayley hyperbolic plane.

\begin{theorem}[Schwartz]\label{Schwartz-theorem}
    Let $X$ be a negatively curved symmetric space other than the real hyperbolic plane $\bH^2$. If $\Gamma$ is a finitely generated group quasi-isometric to a non-uniform lattice $\Lambda$ in $\Isom{X}$, then there exists a short exact sequence
    \[
    \begin{tikzcd}
        1 \ar[r] & F \ar[r] & \Gamma \ar[r] & \Lambda' \ar[r] & 1.
    \end{tikzcd}
    \]
    where $\Lambda' \le \Isom{X}$ is a non-uniform lattice commensurable to $\Lambda$, and $F$ is a finite group.
\end{theorem}

In the case of $X = \bH^2$, non-uniform lattices in $\Isom{\bH^2} = \text{PSL}(2,\bR)$ are virtually the fundamental groups of complete hyperbolic surfaces with finitely many punctures, and hence are virtually free. Any group quasi-isometric to a virtually free group is also virtually free, and every free group can be realized as a non-uniform lattice in $\text{PSL}(2,\bR)$, so quasi-isometric rigidity does hold in this case. However, the additional conclusion of commensurability fails. One reason for this failure is that commensurability preserves arithmeticity, and $\text{PSL}(2,\bR)$ contains both arithmetic and non-arithmetic non-uniform lattices. But since these lattices are virtually free, they are quasi-isometric to each other.

Following Schwartz, lattices such as $\Lambda$ and $\Lambda'$ shall from now on be called \textit{non-uniform rank one lattices}. Then an informal summary might be: any group quasi-isometric to a non-uniform rank one lattice is almost a commensurable non-uniform rank one lattice.

In an extensive study on high-dimensional graph manifolds \cite{FLS}, Frigerio--Lafont--Sisto prove, among many other things, a quasi-isometric rigidity result for products of the form $\pi_1(M) \times \bZ^d$, where $M$ is a complete non-compact finite-volume hyperbolic $m$-manifold, $m\ge3$.

\begin{theorem}[Frigerio--Lafont--Sisto]\label{FLS}
    If $\Gamma$ is a finitely generated group quasi-isometric to $\pi_1(M) \times \bZ^d$, then there exist short exact sequences
    \begin{equation*}
    \begin{tikzcd}
        1 \ar[r] & \bZ^d \ar[r,"j"] & \Gamma' \ar[r] & \Delta \ar[r] & 1
    \end{tikzcd}
    \end{equation*}
    \begin{equation*}
    \begin{tikzcd}
        1 \ar[r] & F \ar[r] & \Delta \ar[r] & \pi_1(M') \ar[r] & 1
    \end{tikzcd}
    \end{equation*}
    where $\Gamma' \le \Gamma$ has finite index, $M'$ is a finite-sheeted covering of $M$, $\Delta$ is a group, and $F$ is a finite group. Moreover, $j(\bZ^d)$ is contained in the center of $\Gamma'$. In other words, $\Gamma$ is virtually a central extension by $\bZ^d$ of a finite extension of $\pi_1(M')$.
\end{theorem}

Observe that the case $d = 0$ is covered by Schwartz' theorem, and indeed, the proof of this theorem applies many of the ideas and results from \cite{Schwartz}.

\subsection{Main results}

Our main contribution is a generalization of Theorem \ref{FLS} to products $\Lambda \times L$, where $\Lambda$ is a non-uniform rank one lattice and $L$ is a lattice in a simply connected nilpotent Lie group. From now on, a lattice such as $L$ is called a \textit{nilpotent lattice}. Our first theorem says that up to finite noise, any group quasi-isometric to $\Lambda \times L$ is an extension of a non-uniform rank one lattice commensurable to $\Lambda$ by a nilpotent lattice quasi-isometric to $L$.

\begin{thmx}\label{Theorem A}
    Let $X \ne \bH^2$ be a negatively curved symmetric space. Let $\Lambda$ be a non-uniform lattice in $\textup{Isom}(X)$ and $L$ be a nilpotent lattice. If $\Gamma$ is a finitely generated group quasi-isometric to $\Lambda \times L$, then there exist short exact sequences
    \begin{equation}\label{virtually nilcentral extension}
    \begin{tikzcd}
        1 \ar[r] & L' \ar[r] & \Gamma' \ar[r] & \Delta \ar[r] & 1,
    \end{tikzcd}
    \end{equation}
    \begin{equation*}
    \begin{tikzcd}
        1 \ar[r] & F \ar[r] & \Delta \ar[r] & \Lambda' \ar[r] & 1.
    \end{tikzcd}
    \end{equation*}
    where $\Gamma' \le \Gamma$ and $\Lambda' \le \Lambda$ have finite index, $L'$ is a nilpotent lattice quasi-isometric to $L$, $\Delta$ is a group, and $F$ is a finite group.
\end{thmx}

The general outline of the proof is similar to that of the proof of Theorem \ref{FLS}. First, the quasi-isometry between $\Gamma$ and $\Lambda \times L$ induces a quasi-action of $\Gamma$ on $B \times N$, where $B \subset X$ is the neutered space associated to $\Lambda$, and $N$ is the simply connected nilpotent Lie group in which $L$ is a lattice. A theorem of Kapovich--Kleiner--Leeb \cite{KKL} guarantees that quasi-isometries $B \times N \to B \times N$ project to quasi-isometries $B \to B$, up to bounded error. Thus the quasi-action of $\Gamma$ on $B\times N$ induces a quasi-action of $\Gamma$ on $B$. A key result in \cite{Schwartz} is that quasi-isometries of the neutered space $B$ have finite distance (with respect to the sup norm) from isometries of $X$. Thus we obtain a homomorphism $\theta : \Gamma \to \text{Isom}(X)$, and we show that the image $\text{im} \hspace{.8mm} \theta$ is a non-uniform lattice commensurable to $\Lambda$. The action $\theta : \Gamma \to \text{Isom}(X)$ came from a quasi-action on $B$, which itself was coarsely projected from a quasi-action on $B \times N$. Hence we are able to show that the kernel of $\theta$ is quasi-isometric to $N$. Then $\Gamma = \text{im} \hspace{.8mm} \theta / \ker \theta$, so we pass to finite-index subgroups as necessary to obtain the desired short exact sequences.

Theorem \ref{FLS} asserts that $\bZ^d$  can be made central in the group extension. In our more general setting, however, the extension is by a nilpotent group. So we define the notion of a nilcentral extension, analogous to that of a central extension, and find conditions which are sufficient to guarantee that (\ref{virtually nilcentral extension}) may be taken to be a nilcentral extension. Given a nilpotent group $G$ with upper central series $1 = Z_0 \lhd Z_1 \lhd \dots \lhd Z_n = G$, define 
\[
    \Sigma(G) \defeq \max_i \text{rank}(Z_{i+1}/Z_i).
\]

\begin{thmx}\label{Theorem B}
    Assume the hypotheses of Theorem \ref{Theorem A} and let $L'$ be the nilpotent lattice obtained from the conclusion of the theorem. If $X$ is either quaternionic hyperbolic space or the Cayley hyperbolic plane, and $\dim\textup{Isom}(X) > \Sigma(L')$, then the group extension (\ref{virtually nilcentral extension}) in the conclusion of Theorem \ref{Theorem A} is virtually nilcentral.
\end{thmx}

Our proof of this theorem relies on a form of Margulis--Corlette--Gromov--Schoen superrigidity for Lie groups with Kazhdan's property (T). The precise statement we apply is in \cite{fisher2012strengthening}.

\section{Preliminaries}

Let $(X, d_X)$ and $(Y, d_Y)$ be metric spaces, and let $k\ge1$ and $c\ge0$ be real numbers. A map $f:X\to Y$ is a $(k,c)$\textit{-quasi-isometric embedding} if for all $a,b \in X$,
\[
\frac{1}{k} d_X(a,b) - c \leq d_Y(f(a), f(b)) \leq k d_X(a,b) + c.
\]
A $(k,c)$-quasi-isometric embedding $f$ is a $(k,c)$-\textit{quasi-isometry} if there is a $(k,c)$-quasi-isometric embedding $g : Y \to X$ such that $d_X(x, (g\circ f)(x)) \le c$ for all $x \in X$, and $d_Y(y, (f\circ g)(y)) \le c$ for all $y\in Y$. Such a map $g$ is a \textit{quasi-inverse} of $f$. For maps $h_1,h_2 : X \to Y$, let $d_Y(h_1,h_2)$ denote $\sup_{x \in X}d_Y(h_1(x),h_2(x))$. If $d_Y(h_1,h_2) < \infty$ then we say that $h_1$ \textit{has finite distance from} $h_2$. With this notation, $g$ is a quasi-inverse of $f$ if and only if $d_X(\text{id}_X, g \circ f) \le c$ and $d_Y(\text{id}_Y, f \circ g) \le c$. A $(k,c)$-quasi-isometric embedding is a $(k',c')$-quasi-isometry for some $k'\ge1$ and $c'\ge 0$ if and only if it is \textit{coarsely surjective}, that is, its image is $r$-dense for some $r \ge 0$. A map $f \colon X \to Y$ is a \textit{quasi-isometry} between $X$ and $Y$ if it is a $(k,c)$-quasi-isometry for some $k \ge 1, c \ge 0$. Two metric spaces $X$ and $Y$ are \textit{quasi-isometric} if there exists a quasi-isometry between them. Observe that a composition of quasi-isometries is again a quasi-isometry.

Let $G$ be a group with a finite, symmetric generating set $S$. For $g \in G$, let $\abs{g}_S$ denote the length of a shortest word representing $g$ with letters in $S$. Then $G$ is endowed with a \textit{word metric} $d_S$ defined by $d_S(g,h) = \abs{g^{-1}h}_S$. Note that left multiplication of $G$ on itself is an isometric action. If $S'$ is another finite, symmetric generating set for $G$, then the identity map $(G,d_S) \to (G,d_{S'})$ is a quasi-isometry (in fact it is bi-Lipschitz). So the word metric on a finitely generated group is well-defined up to quasi-isometry. In other words, the quasi-isometry type of a finitely generated group is well-defined.

Recall the fundamental observation in geometric group theory which equates the quasi-isometry type of a group with the quasi-isometry type of a metric space on which it acts geometrically. Let $X$ be a geodesic metric space which is \textit{proper}, i.e., closed balls are compact. An isometric group action $G \curvearrowright X$ is \textit{properly discontinuous} if for all compact $K \subseteq X$ the set $\s{g \in G : g \cdot K \cap K \ne \emptyset}$ is finite, and \textit{cocompact} if $G\backslash X$ is compact. The action is \textit{geometric} if it is properly discontinuous and cocompact.

\begin{lemma}[Milnor-Schwarz]
If a group $G$ acts geometrically on a proper geodesic metric space $X$, then $G$ is finitely generated and quasi-isometric to $X$. A quasi-isometry $G \mapsto X$ is given by $g \mapsto g \cdot x_0$, where $x_0 \in X$ is any basepoint.
\end{lemma}

For example, if $M$ is a compact Riemannian manifold with Riemannian universal cover $\widetilde M$, then $\pi_1(M)$ acts geometrically on $\widetilde M$ and therefore $\pi_1(M)$ is quasi-isometric to $\widetilde M$.

A more general version of the Milnor-Schwarz lemma exists for quasi-actions, which we now define. Let $(X,d)$ be a geodesic metric space, and let QI($X$) be the set of quasi-isometries $X\to X$ (in contrast to our use, QI($X$) is sometimes used to denote the \textit{quasi-isometry group} of $X$ whose elements are equivalence classes of quasi-isometries). For $k\ge 1$, a \textit{k-quasi-action} of a group $G$ on $X$ is a map $h : G \to \text{QI}(X)$ such that
\begin{enumerate}
    \item For all $g \in G$, $h(g)$ is a $(k,k)$-quasi-isometry with $k$-dense image,
    
    \item $d(h(1),\text{id}_X) \le k$;
    
    \item For all $g_1,g_2 \in G$, $d(h(g_1g_2),h(g_1)h(g_2)) \le k$,
\end{enumerate}
where $h(g_1)h(g_2)$ means $h(g_1) \circ h(g_2)$. A $k$-quasi-action $h$ is \textit{k'-cobounded} if every $G$-orbit is $k'$-dense in $X$. A \textit{(cobounded) quasi-action} is a map which is a ($k'$-cobounded) $k$-quasi-action for some $k\ge1\  (k'\ge1)$. Now we can state a stronger version of the Milnor-Schwarz lemma.

\begin{lemma}[\cite{FLS} Lemma 1.4]
Let $X$ be a geodesic metric space with basepoint $x_0$, and let $G$ be a group. Let $h : G \to \textup{QI}(X)$ be a cobounded quasi-action of $G$ on $X$, and suppose that for each $r > 0$, the set $\s{g \in G : h(g) B(x_0,r) \cap B(x_0,r) \ne \emptyset}$ is finite. Then $G$ is finitely generated and quasi-isometric to $X$. A quasi-isometry $G \to X$ is given by $g \mapsto h(g)(x_0)$.
\end{lemma}

Here are two more lemmas related to quasi-actions that will be useful later.

\begin{lemma}\label{one dense orbit}
    If a quasi-action has at least one dense orbit, then it is cobounded.
\end{lemma}

\begin{proof}
Suppose $h : G \to \text{QI}(X)$ is a $k$-quasi-action such that the $G$-orbit of $x \in X$ is $k'$-dense in $X$. Let $y,p \in X$ be arbitrary, and take $g_1,g_2 \in G$ such that $d(y,h(g_1)(x)) \le k'$ and $d(p,h(g_2)(x)) \le k'$. Then
\begin{align*}
    d(p,h(g_2g_1^{-1})(y))
&\le
    d(p,h(g_2)(x)) 
    +
    d(h(g_2)(x),h(g_2g_1^{-1})(y))
\\&\le
    k' 
    +
    d(h(g_2)(x),h(g_2)h(g_1^{-1})(y))
    +
    d(h(g_2)h(g_1^{-1})(y),h(g_2g_1^{-1})(y))
\\&\le
    k'
    +
    kd(x,h(g_1^{-1})(y))+k
    +
    k
\\&\le
    k'+2k
    +
    k^2d(h(g_1)(x),h(g_1)h(g_1^{-1})(y)) + k^3
\\&\le
    k' + 2k + k^3
    +
    k^2 d(h(g_1)(x), h(1)(y)) 
    +
    k^2 d(h(1)(y), h(g_1)h(g_1^{-1})(y))
\\&\le
    k' + 2k + k^3
    +
    k^2d(h(g_1)(x),y)
    +
    k^2 d(y,h(1)(y))
    +
    k^3
\\&\le
    k'+2k+3k^3+k^2k'.
\end{align*}
So $h$ is $(k'+2k+3k^3+k^2k')$-cobounded. 
\end{proof}

\begin{lemma}\label{induced quasi-action}
    A quasi-isometry between a finitely generated group $G$ and the fundamental group $\pi_1(M)$ of a geodesic metric space induces a cobounded quasi-action of $G$ on the metric universal cover $\widetilde M$. 
\end{lemma}

\begin{proof}
Indeed, since $\pi_1(M)$ is quasi-isometric to $\widetilde M$, we get a quasi-isometry $\phi : G \to \widetilde M$. Let $\psi : \widetilde M \to G$ be a quasi-inverse of $\phi$. For each $g \in G$, define $h(g) : \widetilde M \to \widetilde M$ by
\[
    h(g)(x) = \phi(g\psi(x)), \qquad
    x \in \widetilde M.
\]
Then $h(g)$ is the composition of three quasi-isometries with fixed constants. So $h(g)$ is a quasi-isometry with constants that are independent of $g$. Moreover, $h(1) = \phi \circ \psi$ has finite distance from $\text{id}_{\widetilde M}$. It is also easily checked that $h(g_1g_2)$ has finite distance (bounded independently of $g_1,g_2$) from $h(g_1) \circ h(g_2)$. So $h$ is a quasi-action, and since $\phi$ is coarsely surjective and the left multiplication action of $G$ on itself is transitive, $h$ is cobounded.
\end{proof}

Finally, we recall some facts about non-uniform rank one lattices and their associated neutered spaces. Details may be found in \cite{Schwartz} and \cite{drutu-kapovich}. Let $X$ be a negatively curved symmetric space with distance function $d_X$, and let $\Lambda$ be a non-uniform lattice in Isom($X$). We call such a group a \textit{non-uniform rank one lattice}. By Selberg's lemma \cite{Selberg}, $\Lambda$ is virtually torsion-free. Since we are concerned with the quasi-isometry type of $\Lambda$, we may assume without loss of generality that $\Lambda$ is torsion-free. Then $X/\Lambda$ is a finite-volume non-compact manifold. By truncating the cusps of $X/\Lambda$ we obtain a compact manifold with boundary $Y$, and the universal cover is a \textit{neutered space} $B \subset X$. By definition, $B$ is the complement in $X$ of the union of an infinite equivariant family of open horoballs with pairwise disjoint closures. We refer to the boundaries of these horoballs as the \textit{peripheral horospheres} (or just \textit{horospheres}) of $B$. Since $X/\Lambda$ has finite volume, it has finitely many cusps and $Y$ has finitely many boundary components. It follows that there is an $R > 0$ such that the $d_X$-distance between every pair of distinct horosopheres of $B$ is at least $R$. A horosphere $O$ is centered at a point $p$ in the ideal boundary $\partial X$. We say that $p$ is the \textit{basepoint} of $O$ and that $O$ is \textit{based} at $p$. Since $X/\Lambda$ has finite volume, the basepoints of the horospheres of $B$ are dense in $\partial X$. In general, two horospheres in $X$ share a basepoint if and only if the Hausdorff distance between them is finite.

Since $B$ is a subspace of $X$, it has two natural metrics: the restricted metric $d_X\vert_B$, and the induced length metric $d_B(x,y) \defeq \inf_p\text{length}_{X}(p)$, where the infimum is taken over all paths in $B$ between $x$ and $y$. We will often use the following fact, sometimes without mention.
\begin{lemma}[\cite{drutu-kapovich} Lemma 24.2]\label{unif proper}
    The identity map $(B,d_B) \to (B,d_X\vert_B)$ is 1-Lipschitz and uniformly proper.
\end{lemma}
Recall that a map $f : Z\to W$ between proper metric spaces is \textit{uniformly proper} if $f$ is coarse Lipschitz and there exists a function $\zeta : \bR_+ \to \bR_+$ such that $\text{diam}(f^{-1}(B(w,r))) \le \zeta(r)$ for all $w\in W$, $r > 0$.

\section{Projecting a quasi-action}

Let $X,\Lambda,Y,$ and $B$ be as above. Let $L$ be a lattice in a simply connected nilpotent Lie group $N$. Define $M$ to be the product $Y \times (N/L)$. Let $\Gamma$ be a finitely generated group and assume that $\Gamma$ is quasi-isometric to $\Lambda \times L$. Then by Lemma \ref{induced quasi-action}, the quasi-isometry between $\Gamma$ and $\pi_1(M) = \Lambda \times L$ induces a $k$-cobounded $k$-quasi-action $h$ of $\Gamma$ on $\widetilde{M} = B \times N$ for some $k\ge1$. We denote the metrics on $B\times N, X$, and $B$ by $d, d_X$, and $d_B$, respectively.

We first show that $h$ induces a quasi-action of $\Gamma$ on $B$. In order to do so, we use a theorem of Kapovich--Kleiner--Leeb \cite{KKL} which implies that quasi-isometries of the product $B \times N$ coarsely project to quasi-isometries of $B$. Their theorem is more general than we need, so we only state its application in our situation.

\begin{proposition}\cite[Theorem B]{KKL}
    There exist $D>0$ and $k' \ge 1$, both depending only on $k$, for which the following holds: For each $\gamma \in \Gamma$, there exists a $(k',k')$-quasi-isometry $\psi(\gamma) : B \to B$ with $k'$-dense image such that the diagram 
    \[
    \begin{tikzcd}
        B \times N
        \ar[r,"h(\gamma)"] 
        \ar[d,"\pi"'] 
    &    
        B \times N 
        \ar[d,"\pi"]
    \\    
        B \ar[r,"\psi(\gamma)"'] 
    &    
        B
\end{tikzcd}
\]
commutes up to an error bounded by $D$, where $\pi : B \times N \to B$ is the natural projection.
\end{proposition}

In an effort to make calculations more readable, we denote the value of $\pi$ at $(b,n) \in B \times N$ by $\pi[(b,n)]$ instead of $\pi((b,n))$. 

\begin{lemma}
The map $\psi : \Gamma \to \textup{QI}(B)$ is a $K$-cobounded $K$-quasi-action for some $K \ge 1$.
\end{lemma}

\begin{proof}
By definition, we already know that each $\psi(\gamma)$ is a $(k',k')$-quasi-isometry with $k'$-dense image.  In the following calculations, 1 denotes the identity element of $\Gamma$ and $e$ denotes the identity element of $N$. For all $b \in B$, 
\begin{align*}
    d_B(\psi(1)(b),b)
&\le
    d_B(\psi(1)(b),\pi[h(1)(b,e)]) + d_B(\pi[h(1)(b,e)],\pi[(b,e)])
\\&\le
    D + d(h(1)(b,e),(b,e))
\\&\le
    D+k.
\end{align*}
Let $\gamma_1,\gamma_2 \in \Gamma$ and $b \in B$. Then
\begin{align*}
    d_B(\psi(\gamma_1)\psi(\gamma_2)(b), \psi(\gamma_1\gamma_2)(b))
&\le
    d_B(\psi(\gamma_1)\psi(\gamma_2)(b), \pi[h(\gamma_1)h(\gamma_2)(b,e)]) 
\\&\qquad +
    d_B(\pi[h(\gamma_1)h(\gamma_2)(b,e)], \pi[h(\gamma_1\gamma_2)(b,e)])
\\&\qquad +
    d_B(\pi[h(\gamma_1\gamma_2)(b,e)], \psi(\gamma_1\gamma_2)(b)).
\end{align*}
The first term on the right can be bounded as follows.
\begin{align*}
    d_B(\psi(\gamma_1)\psi(\gamma_2)(b), \pi[h(\gamma_1)h(\gamma_2)(b,e)])
&\le
    d_B(\psi(\gamma_1)\psi(\gamma_2)(b), \psi(\gamma_1)(\pi[h(\gamma_2)(b,e)]))
    \\&\qquad +
    d_B(\psi(\gamma_1)(\pi[h(\gamma_2)(b,e)]), \pi[h(\gamma_1)h(\gamma_2)(b,e)])
\\&\le
    k' d_B(\psi(\gamma_2)(b), \pi[h(\gamma_2)(b,e)]) + k'
    +
    D
\\&\le
    k' D + k' + D.
\end{align*}
Hence,
\begin{align*}
    d_B(\psi(\gamma_1)\psi(\gamma_2)(b), \psi(\gamma_1\gamma_2)(b)) 
&\le
    (k' D + k' + D) 
    + 
    d_B(\pi[h(\gamma_1)h(\gamma_2)(b,e)], \pi[h(\gamma_1\gamma_2)(b,e)])
\\&\qquad +
    d_B(\pi[h(\gamma_1\gamma_2)(b,e)], \psi(\gamma_1\gamma_2)(b))
\\&\le
    k' D + k' + D
    +
    d(h(\gamma_1)h(\gamma_2)(b,e), h(\gamma_1\gamma_2)(b,e)) 
    + 
    D
\\&\le
    k' D + k' + 2D + k.
\end{align*}
Set $K = k' D + k' + 2D + k$. So $\psi$ is a $K$-quasi-action. Let $b_1,b_2 \in B$. Since $h$ is $k$-cobounded, there exists $\gamma \in \Gamma$ such that $d_B( h(\gamma)(b_1,e),(b_2,e)) \le k$. Then
\begin{align*}
    d_B(\psi(\gamma)(b_1),b_2) 
&\le
    d_B(\psi(\gamma)(b_1), \pi[h(\gamma)(b_1,e)])
    +
    d_B(\pi[h(\gamma)(b_1,e)], \pi[(b_2,e)])
\\&\le
    D + d(h(\gamma)(b_1,e),(b_2,e))
\\&\le
    D + k.
\end{align*}
So by possibly increasing $K$, it follows that $\psi$ is $K$-cobounded.

\end{proof}

Next we show that the quasi-action $\psi$ induces a homomorphism $\Gamma \to \text{Isom}(X)$. It is proved in \cite{Schwartz} that the quasi-isometry $\psi(\gamma)$ of $B$ coarsely extends to an isometry of $X$. More precisely, for each $\gamma \in \Gamma$ there exists $\beta>0$ and $\theta(\gamma) \in \text{Isom}(X)$ such that 
\begin{align}\label{isometric extension}
    d_X(\psi(\gamma)(b),\theta(\gamma)(b)) \le \beta
\end{align}
for all $b \in B$. It turns out the constant $\beta$ depends only on the quasi-isometry constants of $\psi(\gamma)$. Hence $\beta$ depends on $K$ and not on $\gamma$. Before we show that $\theta$ is our desired homomorphism, we need the following lemma which states that isometries of $X$ which behave similarly on $B$ must be identical.

\begin{lemma}\label{isometry rigidity}
Let $\alpha,\alpha' \in \textup{Isom}(X)$. If there is a constant $c$ such that $d_X(\alpha(b),\alpha'(b)) \le c$ for all $b \in B$, then $\alpha = \alpha'$.
\end{lemma}

\begin{proof}
By considering $(\alpha')^{-1} \circ \alpha$, we may assume without loss of generality that $\alpha' = \text{Id}_X$. Let $O$ be a peripheral horosphere of $B$ and let $p \in \partial X$ be its basepoint. Let $\s{b_n}$ be a sequence of points in $O$ converging to $p$. Since $d_X(\alpha(b_n),b_n) \le c$ for all $n$, the sequence $\s{\alpha(b_n)}$ also converges to $p$. So the extension of $\alpha$ to $\partial X$ must fix $p$. Since the basepoints of the horospheres of $B$ are dense in $\partial X$ and the extension of $\alpha$ to $\partial X$ is a homeomorphism, $\alpha$ must fix $\partial X$ pointwise. Only the identity map on $X$ extends to the identity map on $\partial X$, so we must have $\alpha = \text{id}_X$.
\end{proof}

\begin{lemma}
The map $\theta : \Gamma \to \textup{Isom}(X)$ is a group homomorphism.
\end{lemma}

\begin{proof}
Let $\gamma_1,\gamma_2 \in \Gamma$. We need to show that $\theta(\gamma_1\gamma_2) = \theta(\gamma_1)\theta(\gamma_2)$. By Lemma \ref{isometry rigidity}, it suffices to produce a constant $c$ such that for all $b \in B$,
\[
    d_X(\theta(\gamma_1\gamma_2)(b),\theta(\gamma_1)\theta(\gamma_2)(b)) \le c.
\]
Since $\psi$ is a $K$-quasi-action, we have that for all $b \in B$,
\begin{align*}
    d_X(\theta(\gamma_1\gamma_2)(b),\theta(\gamma_1)\theta(\gamma_2)(b)) 
&\le
    d_X(\theta(\gamma_1\gamma_2)(b),\psi(\gamma_1\gamma_2)(b)) 
    + 
    d_X(\psi(\gamma_1\gamma_2)(b), \psi(\gamma_1)\psi(\gamma_2)(b))
\\&\qquad +
    d_X(\psi(\gamma_1)\psi(\gamma_2)(b), \theta(\gamma_1)\psi(\gamma_2)(b))
\\&\qquad +
    d_X(\theta(\gamma_1)\psi(\gamma_2)(b), \theta(\gamma_1)\theta(\gamma_2)(b))
\\&\le
    \beta + K + \beta + \beta.
\end{align*}
So set $c = K + 3\beta$.
\end{proof}

Now that we have a homomorphism $\theta : \Gamma \to \text{Isom}(X)$, we can extract information about the algebraic structure of $\Gamma$ by studying the image and kernel of $\theta$.

\section{The image of \texorpdfstring{$\theta$}{theta}}
Let $\Lambda_\Gamma$ denote the image of $\theta$. Our goal is to show that $\Lambda_\Gamma < \text{Isom}(X)$ is a non-uniform lattice that is commensurable with $\Lambda$. Our arguments in this section closely follow those in \cite{FLS}. From now on, unless otherwise stated, the Hausdorff distance $\hdist{\cdot}{\cdot}$ between subsets of $X$ will be considered with respect to $d_X$. 

Recall that a quasi-isometry of the neutered space $B$ coarsely permutes the peripheral horospheres \cite{Schwartz}. The inequality (\ref{isometric extension}) says that the isometry $\theta(\gamma)$ is at finite distance (on $B$) from the quasi-isometry $\psi(\gamma)$. So it follows that $\theta(\gamma)$ also coarsely permutes the peripheral horospheres of $B$. We make this precise in the following lemma. Let $\Omega$ denote the set of peripheral horospheres of $B$.

\begin{lemma}\label{horosphere permutation}
There exists a $\beta'>0$ such that the following holds. For each $\alpha \in \Lambda_\Gamma$ and horosphere $O \in \Omega$, there is a unique horosphere $O' \in \Omega$ such that $\hdist{\alpha(O)}{O'} \le \beta'$.
\end{lemma}

\begin{proof}
Let $\alpha$ and $O$ be given. Then $\alpha = \theta(\gamma)$ for some $\gamma \in \Gamma$. The quasi-isometry $\psi(\gamma)$ permutes the peripheral horospheres of $B$ up to a constant error which depends only on the quasi-isometry constants of $\psi(\gamma)$. Since $\psi$ is a $K$-quasi-action, this means there is a constant $C = C(K)$ and a unique $O' \in \Omega$ such that $\hdist{\psi(\gamma)(O)}{O'} \le C$. From \eqref{isometric extension} we also have $\hdist{\psi(\gamma)(O)}{\theta(\gamma)(O)} \le \beta$. Set $\beta' = \beta + C$. Then $\beta'$ is independent of $\alpha$ and $O$, and
\[
    \hdist{\alpha(O)}{O'} 
\le
    \hdist{\theta(\gamma)(O)}{\psi(\gamma)(O)}
    +
    \hdist{\psi(\gamma)(O)}{O'} 
=
    \beta + C
=
    \beta'.
\]  
\end{proof}

Let $P \subset \partial X$ denote the set of basepoints of the peripheral horospheres of $B$. Since each element of $\Lambda_\Gamma$ coarsely permutes these horospheres, we obtain an action of $\Lambda_\Gamma$ on $P$. 

\begin{lemma}\label{finitely many orbits}
The group $\Lambda_\Gamma$ acts on the set $P$ with finitely many orbits, and given $p \in P$, every element of $\Lambda_\Gamma$ that fixes $p$ under this action also preserves the peripheral horosphere in $\Omega$ based at $p$.
\end{lemma}

\begin{proof}
Let $\alpha \in \Lambda_\Gamma$ and $p \in P$. Let $O \in \Omega$ be the horosphere based at $p$.  By Lemma \ref{horosphere permutation}, there is a unique $O' \in \Omega$ such that $\hdist{\alpha(O)}{O'} < \infty$. Let $p' \in P$ be the basepoint of $O'$. Since the Hausdorff distance between $\alpha(O)$ and $O'$ is finite, the horosphere $\alpha(O)$ is also based at $p'$. So in the extension of $\alpha$ to $\partial X$, we in fact have $\alpha(p) = p'$. Thus $\alpha \cdot p = p'$ defines an action of $\Lambda_\Gamma$ on $ P$.

Next we show that this action has finitely many orbits. Fix a point $b \in B$. For any $r > 0$, $\overline{B_X(b,r)}$ is compact, and since there is a lower bound on the $d_X$-distance between horospheres of $B$, only finitely many of them may intersect $B_X(b,r)$. Hence, we denote by $A \subset \Omega$ the finite subset of horospheres whose $d_X$-distance from $b$ is at most $K + \beta + \beta'$, and let $P_0 \subset P$ be the finite set of basepoints of horospheres in $A$. We show that every orbit in $P$ has a representative in $P_0$. Fix $p \in P$. Let $O \in \Omega$ be the horosphere based at $p$, and let $x \in O$. Since the quasi-action of $\Gamma$ on $B$ is $K$-cobounded, there is a $\gamma \in \Gamma$ such that $d_X(b,\psi(\gamma)(x)) \le d_B(b,\psi(\gamma)(x)) \le K$. Let $\alpha = \theta(\gamma)$. Then $d_X(b,\alpha(x)) \le K + \beta$ by (\ref{isometric extension}). By Lemma \ref{horosphere permutation}, there is a horosphere $O' \in \Omega$ such that the Hausdorff distance between $\alpha(O)$ and $O'$ is at most $\beta'$. In particular, there exists $x' \in O'$ such that $d_X(\alpha(x),x') \le \beta'$. Then
\begin{align*}
    d_X(b,x')
\le
    d_X(b,\alpha(x)) + d_X(\alpha(x),x')
\le
    K + \beta + \beta'.
\end{align*}
So $O' \in A$ and its basepoint $p'$ is in $P_0$. By definition of the action of $\Lambda_\Gamma$ on $P$, we have $\alpha \cdot p = p'$. Hence $p = \alpha^{-1} \cdot p'$, so $p$ is in the orbit of a point in $P_0$. This proves the first part of the statement.

Now let $p \in P$ be the basepoint of $O \in \Omega$ and suppose $\alpha \in \Lambda_\Gamma$ fixes $p$. Since $\alpha \cdot p = p$, the horosphere $\alpha(O)$ is also based at $p$. If $\alpha(O) \ne O$ then $\delta = \hdist{O}{\alpha(O)} > 0$. Moreover, $\alpha$ must be a hyperbolic isometry whose axis has $p$ as an endpoint. So $\hdist{O}{\alpha^n(O)} = n\delta$. Then for large $n$, the Hausdorff distance between $O$ and $\alpha^n(O)$ would exceed $\beta'$, contradicting Lemma \ref{horosphere permutation}. Thus $\alpha(O) = O$.
\end{proof}

Let $P_0 = \s{p_1,\dots,p_k} \subset P$ be a finite set of representatives for the action of $\Lambda_\Gamma$ on $P$, given by Lemma \ref{finitely many orbits}. By removing elements from $P_0$ as necessary, we may assume that no two elements in $P_0$ represent the same orbit. So the orbits of the points in $P_0$ form a partition of $P$. For each $i = 1,\dots,k$, let $O_i \in \Omega$ be the horosphere based at $p_i$, and let $\widehat{O}_i$ be the horosphere contained in the horoball bounded by $O_i$ such that $\hdist{\widehat{O}_i}{O_i} = \beta'$. Let $\widehat{\Omega}$ be the set of horospheres obtained by translating $\widehat{O}_1,\dots,\widehat{O}_k$ by the elements of $\Lambda_\Gamma$. The second part of Lemma \ref{finitely many orbits} guarantees that no two distinct horospheres of $\widehat{\Omega}$ share the same basepoint. Let $\widehat{B}$ denote the complement in $X$ of the union $U$ of the open horoballs bounded by the horospheres in $\widehat{\Omega}$, and endow $\widehat{B}$ with the path-metric induced by $d_X$. Note that $U$ is $\Lambda_\Gamma$-invariant by construction, and so $\widehat{B}$ is $\Lambda_\Gamma$-invariant. It follows that $\Lambda_\Gamma$ acts isometrically on $\widehat{B}$. We observe in the next lemma that $\widehat{B}$ is a thickening of $B$.

\begin{lemma}\label{thick neutered space}
The following containments hold:
\[
    B \subseteq
    \widehat{B} \subseteq
    N_{2\beta'}(B),
\]
where $N_{r}(B) = \s{x \in X : d_X(x,B) \le r}$ for $r > 0$. In particular, $B$ and $\widehat{B}$ are quasi-isometric.
\end{lemma}

\begin{proof}
Let $\widehat{O} \in \widehat{\Omega}$. Then $\widehat{O}$ is based at some $p \in P$, and we let $O$ be the horosphere in $\Omega$ that is based at $p$. By definition, $\widehat{O} = \alpha(\widehat{O}_i)$ for some $i \in \s{1,\dots,k}$ and $\alpha \in \Lambda_\Gamma$. Recall that $\widehat{O}_i$ is the horosphere contained in the horoball bounded by $O_i$ such that $\hdist{\widehat{O}_i}{O_i} = \beta'$. Since $\alpha$ is an isometry, $\alpha(\widehat{O}_i)$ must be the horosphere contained in the horoball bounded by $\alpha(O_i)$ such that $\hdist{\alpha(\widehat{O}_i)}{\alpha(O_i)} = \beta'$. By Lemma \ref{horosphere permutation}, $\hdist{\alpha(O_i)}{O} \le \beta'$. Altogether, $\widehat{O}$ must be contained in the horoball bounded by $O$, and
\begin{align*}
    \hdist{O}{\widehat{O}}
\le
    \hdist{O}{\alpha(O_i)} +
    \hdist{\alpha(O_i)}{\alpha(\widehat{O}_i)}
\le
    2\beta'.
\end{align*}
The first part of the lemma follows, and so the inclusion $B \hookrightarrow \widehat{B}$ is a quasi-isometry.
\end{proof}

Since $B/\Lambda$ has finitely many boundary components which lift to the peripheral horospheres, there is an $R > 0$ such that the $d_X$-distance between distinct horospheres of $B$ is at least $R$. Then $B \subseteq \widehat{B}$ implies that the $d_X$-distance between distinct elements of $\widehat{\Omega}$ is also at least $R$. We are now equipped to show that $\Lambda_\Gamma$ is a non-uniform lattice.

\begin{proposition}
The group $\Lambda_\Gamma$ is a non-uniform lattice in $\textup{Isom}(X)$, and admits $\widehat{B}$ as its associated neutered space.
\end{proposition}

\begin{proof}
We first show that $\Lambda_\Gamma$ is a discrete subgroup. Recall that $P \subset \partial X$ is a dense subset. So pick horospheres $\widehat{O}_1,\dots,\widehat{O}_{m+1}$ in $\widehat{\Omega}$, where $m = \dim X$, with basepoints $p_1,\dots,p_{m+1}$ such that $\s{p_1,\dots,p_{m+1}}$ is not contained in the ideal boundary of any hyperplane of $X$. Then the only isometry of $X$ whose extension to $\partial X$ fixes every $p_i$ is the identity map. For each $i=1,\dots,m+1$, pick some $x_i \in \widehat{O}_i$, and set
\[
    U = \s{\alpha \in \text{Isom}(X) : d_X(\alpha(x_i),x_i) < R \text{ for } i = 1,\dots,m+1}.
\]
Then $U \subset \text{Isom}(X)$ is an open neighborhood of the identity, with respect to the compact-open topology. Let $\alpha \in \Lambda_\Gamma \cap U$. Since $\alpha \in \Lambda_\Gamma$, $\alpha$ permutes $\widehat{\Omega}$. Since $\alpha \in U$, $d_X(\alpha(\widehat{O}_i),\widehat{O}_i) < R$. Together, these imply that $\alpha(\widehat{O}_i) = \widehat{O}_i$, and hence $\alpha \cdot p_i = p_i$, for each $i=1,\dots,m+1$. So $\alpha = \text{id}_X$ and $\Lambda_\Gamma \cap U = \s{\text{id}_X}$. Thus $\Lambda_\Gamma$ is a discrete subgroup.

Next we show that $X/\Lambda_\Gamma$ has finite volume. Recall that the quasi-action of $\Gamma$ on $B$ is $K$-cobounded. It follows from the definition of $\Lambda_\Gamma$ and the containment $\widehat{B} \subseteq N_{2\beta'}(B)$ that the $\Lambda_\Gamma$-orbit of a (every) point in $\widehat{B}$ is $(K+3\beta')$-dense in $\widehat{B}$. Hence the quotient orbifold $\widehat{B}/\Lambda_\Gamma$ is compact, and by Lemma \ref{finitely many orbits}, it has finitely many boundary components $V_1,\dots,V_j$. Suppose $\widehat{O}_i \in \widehat{\Omega}$ projects onto the boundary component $V_i$. Note that $\alpha(\widehat{O}_i) \cap \widehat{O}_i \ne \emptyset$ if and only if $\alpha(\widehat{O}_i) = \widehat{O}_i$ if and only if $\alpha$ is in the stabilizer ${\Lambda_\Gamma}_i$ of the basepoint of $\widehat{O}_i$. The subset $V_i = \widehat{O}_i/{\Lambda_\Gamma}_i$ of the compact quotient $\widehat{B}/\Lambda_\Gamma$ is closed, and therefore, compact. Let $W_i \subset X$ be the horoball bounded by $\widehat{O}_i$. Then $W_i/{\Lambda_\Gamma}_i$ has finite volume, and since $\parens{\bigcup_{i=1}^j W_i} \cup \widehat{B}$ projects surjectively onto $X/\Lambda_\Gamma$, we conclude that $X/\Lambda_\Gamma$ has finite volume. Thus $\Lambda_\Gamma$ is a non-uniform lattice.
\end{proof}

We are now in a position to use Schwartz' rigidity theorem for non-uniform rank one lattices.

\begin{corollary}\label{commensurable}
The group $\Lambda_\Gamma$ is commensurable with $\Lambda$.
\end{corollary}

\begin{proof}
Recall that $B$ and $\widehat{B}$ are quasi-isometric by Lemma \ref{thick neutered space}. Since $\Lambda$ acts geometrically on $B$ and $\Lambda_\Gamma$ acts geometrically on $\widehat{B}$, the Milnor-Schwarz lemma implies that the two non-uniform lattices $\Lambda$ and $\Lambda_\Gamma$ are quasi-isometric. Then the corollary follows from Schwartz' theorem.
\end{proof}

\section{The kernel of \texorpdfstring{$\theta$}{theta}}
Next we analyze the kernel of $\theta$, and our goal is to show that $\ker \theta$ is virtually a nilpotent lattice. To this end, we first show that via the quasi-action $h$ of $\Gamma$ on $B \times N$, each $\gamma \in \ker\theta$ coarsely preserves the fibers $\s{b} \times N \subset B \times N$. 

\begin{lemma}\label{fiber stability}
There is a $\delta > 0$ such that for each $\gamma \in \ker\theta$ the following holds. For every fiber $F = \s{b} \times N \subset B \times N$ and every $x \in F$, we have $h(\gamma)(x) \in N_\delta(F)$
\end{lemma}

\begin{proof}
Fix $\gamma \in \ker \theta$. Then $\theta(\gamma) = \text{id}_X$ and so $d_X(\psi(\gamma)(b),b) \le \beta$ for all $b \in B$. By Lemma \ref{unif proper}, there exists $\delta' > 0$ so that in $B$, every $d_X$-ball of radius $\beta$ is contained in a $d_B$-ball of radius $\delta'$. Then $d_B(\psi(\gamma)(b),b) \le \delta'$ for all $b \in B$. Take a fiber $F = \s{b_0} \times N \subset B \times N$ and let $x \in F$. Then
\begin{align*}
    d_B(\pi[h(\gamma)(x)],b_0)
&\le
    d_B(\pi[h(\gamma)(x)], \psi(\gamma)(\pi[x]))
    +
    d_B(\psi(\gamma)(b_0),b_0)
\le
    D + \delta'.
\end{align*}
Set $\delta = D + \delta'$ and note that this value does not depend on the choice of $\gamma$. Then observe that $d_B(\pi[h(\gamma)(x)],b_0) \le \delta$ if and only if $h(\gamma)(x) \in N_\delta(F)$.
\end{proof}

Fix a fiber $F_0 = \s{b_0} \times N \subset B \times N$. For $\gamma \in \ker\theta$ and $x \in F_0$, Lemma \ref{fiber stability} guarantees that there exists a point $f(\gamma)(x)$ in $F_0$ such that $d(f(\gamma)(x),h(\gamma)(x)) \le \delta$. So for each $\gamma \in \ker \theta$ we have a map $f(\gamma) : F_0 \to F_0$.

\begin{proposition}\label{kernel action}
The map $\gamma \mapsto f(\gamma)$ is a cobounded quasi-action of $\ker\theta$ on $F_0$.
\end{proposition}

\begin{proof}
First we show that $f(\gamma)$ is a quasi-isometry.
Fix $\gamma \in \ker\theta$, and let $x,y \in F_0$. Then
\begin{align*}
    d(f(\gamma)(x),f(\gamma)(y))
&\le
    d(h(\gamma)(x),h(\gamma)(y)) + 2\delta
\\&\le
    kd(x,y) + k + 2\delta
\end{align*}
and
\begin{align*}
    d(f(\gamma)(x),f(\gamma)(y)) 
&\ge
    d(h(\gamma)(x),h(\gamma)(y)) - 2\delta
\\&\ge
    \frac{1}{k}d(x,y) - k - 2\delta
\end{align*}
Let $z \in F_0$. Since $h(\gamma)$ has $k$-dense image, there is some $y = (b,n) \in B \times N$ for which $d(z,h(\gamma)(y)) \le k$. By Lemma \ref{fiber stability}, $d_B(\pi[h(\gamma)(y)],b) \le \delta$. So it turns out that $b$ cannot be too far from $b_0$. Indeed,
\begin{align*}
    d_B(b_0,b)
&\le
    d_B(b_0,\pi[h(\gamma)(y)]) + d_B(\pi[h(\gamma)(y)],b)
\\&\le
    d(z,h(\gamma)(x)) + \delta
\\&\le
    k + \delta.
\end{align*}
Set $x = (b_0,n)$. Then $d(x,y) = d_B(b_0,b) \le k + \delta$, and so
\begin{align*}
    d(z,f(\gamma)(x))
&\le
     d(z,h(\gamma)(y)) 
    +
    d(h(\gamma)(y), h(\gamma)(x))
    +
    d(h(\gamma)(x),f(\gamma)(x))
\\&\le
    k + (kd(x,y) + k) + \delta
\\&\le
    k + k(k+\delta) + k + \delta.
\end{align*}
It follows that $f(\gamma)$ is a quasi-isometry.

For all $x \in F_0$,
\[
    d(f(1)(x),x) \le
    d(f(1)(x),h(1)(x)) + d(h(1)(x),x) \le
    \delta + k.
\]
Let $\gamma_1,\gamma_2 \in \Gamma$ and $x \in F_0$. Then
\begin{align*}
    d(f(\gamma_1\gamma_2)(x),f(\gamma_1)f(\gamma_2)(x))
&\le 
    d(f(\gamma_1\gamma_2)(x),h(\gamma_1\gamma_2)(x))
    +
    d(h(\gamma_1\gamma_2)(x),f(\gamma_1)f(\gamma_2)(x))
\\&\le
    \delta 
    +
    d(h(\gamma_1\gamma_2)(x),h(\gamma_1)h(\gamma_2)(x))
\\&\qquad +
    d(h(\gamma_1)h(\gamma_2)(x),f(\gamma_1)f(\gamma_2)(x))
\\&\le
    \delta + k
    +
    d(h(\gamma_1)h(\gamma_2)(x),h(\gamma_1)f(\gamma_2)(x))
\\&\qquad+
    d(h(\gamma_1)f(\gamma_2)(x),f(\gamma_1)f(\gamma_2)(x))
\\&\le
    \delta+k
    +
    [kd(h(\gamma_2)(x),f(\gamma_2)(x))+k]
    +
    \delta
\\&\le
    \delta + k + k\delta + k + \delta.
\end{align*}
It follows that $f$ is a quasi-action, and it remains to show that $f$ is cobounded. 

Recall that it suffices to show that the orbit of $x_0 = (b_0,e) \in F_0$ is $r$-dense in $F_0$ for some $r > 0$. First we claim that if $\gamma \in \Gamma$ is such that $h(\gamma)(x_0) \in N_k(F_0)$, then $\theta(\gamma)$ moves $b_0$ by a distance which is bounded independently of $\gamma$. Indeed, if $h(\gamma)(x_0) \in N_k(F_0)$ then 
\begin{align*}
    d_X(\theta(\gamma)(b_0),b_0) &\le
    d_X(\theta(\gamma)(b_0),\psi(\gamma)(b_0)) + d_B(\psi(\gamma)(\pi[x_0]),\pi[h(\gamma)(x_0)]) + d_B(\pi[h(\gamma)(x_0)],\pi[x_0])
    \\&\le
    \beta + D + k.
\end{align*}
By an Arzel\`a-Ascoli theorem, the isometries of $X$ which move $b_0$ by a distance at most $\beta+D+k$ lie in a compact subset $C$ of the compact-open topology on $\text{Isom}(X)$. Since $\Lambda_\Gamma \subset \text{Isom}(X)$ is discrete, the intersection $A = \Lambda_\Gamma \cap C$ is finite. In summary, $h(\gamma)(x_0) \in N_k(F_0)$ implies $\theta(\gamma) \in A$. For each $a \in A$, pick $\gamma_a \in \theta^{-1}(\s{a})$ and set $M = \max_{a \in A} d(x_0,h(\gamma_a^{-1})(x_0))$.

Now, let $y \in F_0$ be arbitrary. Since the quasi-action $h$ of $\Gamma$ on $B \times N$ is $k$-cobounded, there is a $\gamma \in \Gamma$ such that $d(y,h(\gamma)(x_0)) \le k$. Then $h(\gamma)(x_0) \in N_k(F_0)$, so $\theta(\gamma)$ is equal to some $a \in A$. Then $\gamma\gamma_a^{-1} \in \ker\theta$ and
\begin{align*}
    d(y,h(\gamma\gamma_a^{-1})(x_0))
&\le
    d(y,h(\gamma)h(\gamma_a^{-1})(x_0)) + k
\\&\le
    d(y,h(\gamma)(x_0)) + d(h(\gamma)(x_0),h(\gamma)h(\gamma_a^{-1})(x_0)) + k
\\&\le
    k d(x_0,h(\gamma_a^{-1})(x_0)) + 3k
\\&\le
    k M + 3k.
\end{align*}
Hence, $d(y,f(\gamma\gamma_a^{-1})(x_0)) \le k M + 3k + \delta$, and so the $(\ker\theta)$-orbit of $x_0$ is $(k M + 3k + \delta)$-dense in $F_0$.

\end{proof}

\begin{proposition}\label{kernel qi to N}
The group $\ker\theta$ is finitely generated and quasi-isometric to $N$.
\end{proposition}

\begin{proof}
Set $x_0 = (b_0,e) \in F_0$. Since $f$ is a cobounded quasi-action of $\ker\theta$ on $F_0$, and $F_0$ is isometric to $N$, the proposition follows from the stronger version of the Milnor-Schwarz lemma once we show that for $r > 0$, the set $\s{\gamma \in \ker\theta : f(\gamma)B_{F_0}(x_0,r) \cap B_{F_0}(x_0,r) \ne \emptyset}$ is finite. By definition, $d(f(\gamma)(x),h(\gamma)(x)) \le \delta$ for $\gamma \in \ker\theta$ and $x \in F_0$. So it suffices to show that for $r > 0$, the set $\s{\gamma \in \Gamma : h(\gamma)B(x_0,r) \cap B(x_0,r) \ne \emptyset}$ is finite, where $B(x_0,r) = B_{B\times N}(x_0,r)$. Recall that $h$ is defined by $h(\gamma)(x) = \phi(\gamma\phi^{-1}(x))$ for $\gamma \in \Gamma$ and $x \in B \times N$, where $\phi : \Gamma \to B \times N$ is a quasi-isometry and $\phi^{-1}$ is a quasi-inverse. Since $\phi$ is a quasi-isometry, the set $\s{\gamma' \in \Gamma : \phi(\gamma') \in B(x_0,r)}$ is finite. Since $\phi^{-1}$ is a quasi-isometry, $\phi^{-1}(B(x_0,r))$ is also finite. It follows that the set $\s{\gamma \in \Gamma : \phi(\gamma\phi^{-1}(B(x_0,r))) \cap B(x_0,r) \ne \emptyset}$ must be finite, as desired.
\end{proof}

\section{Quasi-isometric rigidity and nilcentral extensions}
We are now ready to prove our main result which is quasi-isometric rigidity for products of the form $\Lambda \times L$, where $\Lambda$ is a non-uniform rank one lattice and $L$ is a nilpotent lattice.

\begin{theorem}
Let $X \ne \bH^2$ be a negatively curved symmetric space. Let $\Lambda$ be a non-uniform lattice in $\textup{Isom}(X)$ and $L$ be a nilpotent lattice. If $\Gamma$ is a finitely generated group quasi-isometric to $\Lambda \times L$, then there exist short exact sequences
    \begin{equation}
    \begin{tikzcd}
        1 \ar[r] & L' \ar[r] & \Gamma' \ar[r] & \Delta \ar[r] & 1,
        \tag{\ref{virtually nilcentral extension}}
    \end{tikzcd}
    \end{equation}
    \begin{equation*}
    \begin{tikzcd}
        1 \ar[r] & F \ar[r] & \Delta \ar[r] & \Lambda' \ar[r] & 1.
    \end{tikzcd}
    \end{equation*}
    where $\Gamma' \le \Gamma$ and $\Lambda' \le \Lambda$ have finite index, $L'$ is a nilpotent lattice quasi-isometric to $L$, $\Delta$ is a group, and $F$ is a finite group.
\end{theorem}

\begin{proof}
By Proposition \ref{kernel qi to N}, $\ker\theta$ is quasi-isometric to $N$, where $N$ is the simply connected nilpotent Lie group in which $L$ is a lattice. In particular $\ker\theta$ has polynomial growth. Then by Gromov's polynomial growth theorem, $\ker\theta$ is virtually nilpotent. So $\ker\theta$ has a finite-index, hence finitely generated, nilpotent subgroup $K$. Such groups are virtually torsion-free, so take a finite-index torsion-free subgroup $K' \le K$. Since $\ker\theta$ is finitely generated, it has only finitely many subgroups with the same index as $K'$. Let $L'$ denote their intersection. Then $L'$ is a finite-index characteristic subgroup of $\ker\theta$. Moreover, $L'$ is finitely generated, nilpotent, and torsion-free. A theorem of Malcev \cite{malcev} then says that $L'$ embeds as a lattice in a simply connected nilpotent Lie group $N'$, known as the (real) Malcev completion of $L'$ (see also \cite[Theorem 2.18]{raghunathan1972discrete}). Since $L'$ is characteristic in $\ker\theta$, it is normal in $\Gamma$. By construction, $\Gamma/L'$ is an extension of $\Lambda_\Gamma = \Gamma/\ker\theta$ by the finite group $F = \ker\theta/L'$. By Corollary \ref{commensurable}, $\Lambda_\Gamma$ is virtually a finite-index subgroup $\Lambda' \le \Lambda$ which is also a non-uniform lattice in $\text{Isom}(X)$. Let $\Gamma' = \theta^{-1}(\Lambda')$ and $\Delta = \Gamma'/L'$. Then we obtain the desired short exact sequences.

Since lattices in simply connected nilpotent Lie groups are uniform, $L, N, \ker\theta$, and $L'$ are all quasi-isometric to each other. The quotient $\Gamma'/L'$ is a finite extension of a finite-index subgroup of $\Lambda$, so $\Gamma'/L'$ and $\Lambda$ are virtually isomorphic, hence quasi-isometric, to each other.
\end{proof}

\begin{remark}
One may wonder whether the fact that $N$ and $N'$ are quasi-isometric determines any algebraic relation between them or between $L$ and $L'$. However, it is an open problem whether quasi-isometric simply connected nilpotent Lie groups are necessarily isomorphic.
\end{remark}

\begin{remark}
It is natural to wonder whether the short exact sequence (\ref{virtually nilcentral extension}) splits, perhaps after passing to further finite-index subgroups. Indeed, a conclusion such as this would strengthen even further the connection between the algebraic structure of $\Gamma$ and that of $\Lambda \times L$. However, it was shown in \cite{FLS} that the sequence does not in general virtually split when $\Lambda$ is the fundamental group of a complete finite-volume real hyperbolic $m$-manifold, $m\ge 3$, and $N = \bZ^d$.
\end{remark}

In the setting of Theorem \ref{FLS} where $L'=\bZ^d$, $\Gamma'$ may be chosen so that the group extension (\ref{virtually nilcentral extension}) is a central extension. In our more general setting where $L'$ is not necessarily abelian, it does not make sense to ask whether (\ref{virtually nilcentral extension}) is a central extension. Despite this, we are still able to derive some extra structure for the group extension (\ref{virtually nilcentral extension}), analogous to centrality, when we focus our attention to the case when $X$ is either quaternionic hyperbolic space or the Cayley hyperbolic plane. Our argument works for these specific symmetric spaces because we rely on a superrigidity theorem that applies when $\text{Isom}(X)$ has Kazhdan's property (T). Indeed, for those symmetric spaces, $\Isom{X}$ is either $\text{Sp}(n,1)$ or $\text{F}_{4(-20)}$, both of which are known to have property (T) by Kostant \cite{kostant} (see also \cite{BdlHV}). On the other hand, if $X$ is real or complex hyperbolic space, then $\Isom{X}$ is $\text{SO}(n,1)$ or $\text{SU}(n,1)$ and does not have property (T).

To state the superrigidity theorem, we first give a definition. Let $V$ be a finite-dimensional vector space over $\bR$ or $\bC$. Given a lattice $\Lambda$ in a locally compact group $G$ and a representation $\pi : \Lambda \to \text{GL}(V)$, we say that $\pi$ \textit{almost extends to a continuous representation of} $G$ if there exist representations $\pi_1 : G \to \text{GL}(V)$ and $\pi_2 : \Lambda \to \text{GL}(V)$ such that: $\pi_2$ has bounded image, the images of $\pi_1$ and $\pi_2$ commute, and $\pi(\gamma) = \pi_1(\gamma)\pi_2(\gamma)$ for all $\gamma \in G$.

\begin{theorem}\cite[Theorem 3.7]{fisher2012strengthening}\label{superrigidity}
Let $G$ be a semisimple Lie group with no compact factors and Kazhdan's property (T). Let $\Lambda$ be a lattice in $G$ and $\pi : \Lambda \to \textup{GL}(V)$ a finite-dimensional representation. Then $\pi$ almost extends to a continuous representation of $G$.
\end{theorem}

We now give a precise definition of the extra structure we alluded to above. Let $G$ be a group with a nilpotent normal subgroup $N$. The upper central series of $N$ is 
\[
    1 = Z_0 \lhd Z_1 \lhd \dots \lhd Z_n = N
\]
where $Z_1 = Z(N)$ and $Z_{i+1}/Z_i = Z(N/Z_i)$. Recall that the $Z_i$ are characteristic subgroups of $N$, so in particular they are normal in $G$. We say that $N$ is a \textit{nilcentral} subgroup of $G$ if $Z(N/Z_i) \subseteq Z(G/Z_i)$ for each $i=0,\dots,n$. Observe that $N$ is nilcentral in $G$ if and only if $Z(N) \subseteq Z(G)$ and $N/Z(N)$ is nilcentral in $G/Z(N)$. An abelian subgroup is nilcentral if and only if it is central, and a nilpotent group $N$ is nilcentral in $N\times G$ for every group $G$. We now give an example of a group with a normal subgroup which is nilpotent but not nilcentral.

\begin{example}\label{non-nilcentral subgroup}
Let $N$ be the discrete Heisenberg group with group presentation 
\[
    N = \angles{a,b,c\ |\ [a,c], [b,c],
    [a,b]c^{-1}}. 
\]
Let $H = \bZ_2 = \s{\pm 1}$ act on $N$ via the action $\phi : H \to \text{Aut}(N)$ defined by
\[
    \phi_{-1}(a) = a, \quad
    \phi_{-1}(b) = b, \quad
    \phi_{-1}(c) = c^{-1}.
\]
Let $G = N \rtimes_\phi H$. Then $N$ is a nilpotent normal subgroup of $G$ but it is not nilcentral because $Z(N) \not\subseteq Z(G)$. Indeed,
\[
    (c,1)(a,-1) = (ca,-1) \qquad
    \text{and} \qquad
    (a,-1)(c,1) = (ac^{-1},-1),
\]
but $(ca,-1) = (ac^{-1},-1) = (c^{-1}a,-1)$ if and only if $c = c^{-1}$, which is not the case. Thus $c \in Z(N)$, but $(c,1) \notin Z(G)$.
\end{example}

Let
\begin{equation}\label{ses0}
\begin{tikzcd}
    1 \ar[r] & N \ar[r,"j"] & G \ar[r] & H \ar[r] & 1
\end{tikzcd}
\end{equation}
be a short exact sequence of groups, where $N$ is nilpotent. We say that (\ref{ses0}) is a \textit{nilcentral} extension if $j(N)$ is nilcentral in $G$. Note that an abelian extension is nilcentral if and only if it is central. Our theorem at the end of the section provides a family of non-trivial examples of nilpotent extensions that are nilcentral. On the other hand, if $\phi : H \to \text{Aut}(N)$ is as in Example \ref{non-nilcentral subgroup}, then
\[
\begin{tikzcd}
    1 \ar[r] & N \ar[r] & N \rtimes_\phi H \ar[r] & H \ar[r] & 1
\end{tikzcd}
\]
is a nilpotent extension which is not nilcentral.

We will see that in the situation of Theorem \ref{Theorem A} under some additional conditions, passing to a finite-index subgroup of $\Gamma'$ does yield a nilcentral extension. To this end, we make the following definition. We say that \eqref{ses0} is a \textit{virtually nilcentral} extension if $G$ has a finite-index subgroup $G'$ containing $j(N)$ such that $j(N)$ is nilcentral in $G'$. Recall that this means $Z(N/Z_i) \subseteq Z(G'/Z_i)$ for $i=0,\dots,n$, where $N$ is identified with $j(N) \subseteq G$. We now give a more tractable characterization of virtually nilcentral extensions.

\begin{proposition}\label{virtually nilcentral characterization}
The short exact sequence \eqref{ses0} is a virtually nilcentral extension if and only if for each $i=0,\dots,n$, $G/Z_i$ has a finite-index subgroup $K_i$ containing $N/Z_i$, such that $Z(N/Z_i) \subseteq Z(K_i)$.
\end{proposition}

\begin{proof}
The forward direction is done by taking $K_i = G'/Z_i$ for each $i$. For the reverse direction, recall that $K_i = G_i/Z_i$ for some $G_i \le G$. So $N \le G_i$ and $[G:G_i] = [G/Z_i : K_i] < \infty$. Let $G' = \bigcap G_i$. Since each $G_i$ has finite index, so does $G'$. For each $i$, $N/Z_i \subset G'/Z_i \subset K_i$, and since $Z(N/Z_i) \subset Z(K_i)$, we have $Z(N/Z_i) \subset Z(G'/Z_i)$.
\end{proof}

Now suppose we are in the situation of Theorem \ref{Theorem A} and we have the following two short exact sequences.
\begin{equation}\label{vne}
\begin{tikzcd}
    1 \ar[r] & N \ar[r] & \Gamma \ar[r] & \Delta \ar[r] & 1
\end{tikzcd}
\end{equation}
\begin{equation}\label{finite extension}
\begin{tikzcd}
    1 \ar[r] & F \ar[r] & \Delta \ar[r] & \Lambda \ar[r] & 1
\end{tikzcd}
\end{equation}
where $N$ is a nilpotent lattice (we are now using $N$ to denote the lattice, not the ambient Lie group), $F$ is finite, and $\Lambda$ is a non-uniform lattice in $\text{Isom}(X)$, where $X \ne \bH^2$ is a negatively curved symmetric space. Since $N$ is finitely generated, nilpotent, and torsion-free, $Z(N/Z_i) = Z_{i+1}/Z_i$ is a finitely generated free abelian group. That is, $Z(N/Z_i) = \bZ^{d_i}$ for some $d_i$, and so $\Sigma(N) = \max_i d_i$. We now focus on the situation when $X$ is quaternionic hyperbolic space or the Cayley hyperbolic plane, and $\text{Isom}(X)$ has sufficiently large dimension.

\begin{theorem}
Suppose $X$ is either quaternionic hyperbolic space or the Cayley hyperbolic plane. If $\dim \textup{Isom}(X) > \Sigma(N)$, then \eqref{vne} is a virtually nilcentral extension.
\end{theorem}

\begin{proof} 
Fix $i \in \s{0,\dots,n}$. By Proposition \ref{virtually nilcentral characterization} it suffices to show that $\Gamma/Z_i$ has a finite-index subgroup $K_i$ containing $N/Z_i$ in which $Z(N/Z_i)$ is central. Since $N/Z_i$ is normal in $\Gamma/Z_i$ and $Z(N/Z_i)$ is characteristic in $N/Z_i$, we have $Z(N/Z_i) \lhd \Gamma/Z_i$. So $\Gamma/Z_i$ acts by conjugation on $Z(N/Z_i)$, and we call this action $\rho_1$. If $\rho_1:\Gamma/Z_i \to \text{Aut}(Z(N/Z_i))$ has finite image, then we may take $K_i = \ker\rho_1$ to finish. Since $N/Z_i \le \ker\rho_1$, we may project $\rho_1$ to an action $\rho_2$ of $\Delta = (\Gamma/Z_i)/(N/Z_i)$ on $Z(N/Z_i)$. Then it suffices to show that $\rho_2 : \Delta \to \text{Aut}(Z(N/Z_i))$ has finite image because $\text{im }\rho_2 = \text{im }\rho_1$. 

Recall that $Z(N/Z_i) = \bZ^d$ for some $d \le r$. So $\text{Aut}(Z(N/Z_i)) = \text{GL}(d,\bZ)$, and the latter group naturally embeds into $\text{GL}(d,\bR)$. Composing $\rho_2$ with this inclusion yields a homomorphism $\rho_3 : \Delta \to \text{GL}(d,\bR)$, and it suffices to show that $\rho_3$ has bounded image. From \eqref{finite extension}, $\Delta$ is finitely generated because $F$ is finite and the lattice $\Lambda$ is finitely generated. Thus by Selberg's lemma, $\rho_3(\Delta)$ has a finite-index subgroup which is torsion-free. This subgroup has pre-image $\Delta' \le \Delta$ which also has finite index. Let $\rho_4 : \Delta' \to \text{GL}(d,\bR)$ be the restriction $\rho_3\vert_{\Delta'}$, and note that since $[\Delta : \Delta'] < \infty$, it suffices to show that $\rho_4$ has bounded image. Since $\rho_4(\Delta')$ is torsion-free and $F$ is finite, we have that $\Delta' \cap F \subset \ker\rho_4$. So $\rho_4$ projects to a homomorphism $\rho_5 : \Delta'/(\Delta'\cap F) \to \text{GL}(d,\bR)$. Since $\text{im}\ \rho_5 = \text{im}\ \rho_4$, it suffices to show that $\rho_5$ has bounded image. 

Now, $\Delta'/(\Delta' \cap F)$ is isomorphic to $(\Delta'F)/F$, and since $\Delta'$ has finite index in $\Delta$, $(\Delta'F)/F$ has finite index in $\Delta/F = \Lambda$. Thus, $\Delta'/(\Delta'\cap F)$ is a lattice in $\text{Isom}(X)$. The hypothesis on $X$ implies $\text{Isom}(X)$ has property (T). Hence, we may apply Theorem \ref{superrigidity} to $\rho_5$ to obtain representations $\pi_1 : \text{Isom}(X) \to \text{GL}(d,\bR)$ and $\pi_2 : \Delta'/(\Delta'\cap F) \to \text{GL}(d,\bR)$, where $\pi_2$ has bounded image, such that $\rho_5(\gamma) = \pi_1(\gamma)\pi_2(\gamma)$ for $\gamma \in \Delta'/(\Delta'\cap F)$. If $\dim \text{Isom}(X) > \Sigma(N) \ge d$, so that the simple group $\text{Isom}(X)$ has larger dimension than $\text{GL}(d,\bR)$, then $\pi_1$ must be trivial, in which case $\rho_5 = \pi_2$ has bounded image.
\end{proof}

\begin{remark}
    $\dim \text{Sp}(n,1) = 2n^2 + 5n + 3$ and $\dim \text{F}_{4(-20)} = 52$.
\end{remark}

\bibliographystyle{alpha}
\bibliography{references.bib}

\end{document}